\documentclass[12pt]{article}

\usepackage{amssymb}
\usepackage{graphics}

\newcommand{\qed}{$\;\;\;\Box$}
\newenvironment{proof}{\par\smallbreak{\sl\bf Proof.~}}
{\unskip\nobreak\hfill \qed \par\medbreak}

\newcounter{claim}
\renewcommand{\theclaim}{\arabic{claim}}
{\par\medskip\par}

{\qed\par\smallbreak}
\newcommand{\N}{{\mathbb N}}
\newcommand{\R}{{\mathbb R}}

\newcommand{\Z}{{\mathbb Z}}

\newcommand{\Q}{{\mathbb Q}}

\newcommand{\CC}{{\cal C}}
\newcommand{\KK}{{\cal K}}

\newcommand{\Cv}{\CC_{n}}

\newcommand{\beq}{\begin{equation}}
\newcommand{\ee}{\end{equation}}

\renewcommand{\d}{\partial}

\newtheorem{thm}{Theorem}[section]
\newtheorem{lem}[thm]{Lemma}

\newtheorem{defn}[thm]{Definition}

\newtheorem{rem}[thm]{Remark}

\newcommand{\ga}{\gamma}

\newcommand{\vphi}{\varphi}

\newcommand{\om}{\tau}

\newcommand{\reff}[1]{(\ref{#1})}      

\title{
Periodic Solutions to Dissipative Hyperbolic Systems. I: Fredholm Solvability of Linear Problems 
} 

\newcounter{thesame}
\author{
Irina Kmit
 \ \ \ Lutz Recke\\
{\small
Institute of Mathematics, Humboldt University of Berlin,}
\\
{\small Rudower Chaussee 25, D-12489 Berlin, Germany }
\\
{\small
and Institute for Applied Problems of Mechanics and Mathematics, }
\\
{\small
Ukrainian Academy of Sciences,  Naukova St.\ 3b, 79060 Lviv,
Ukraine 
}
\\
{\small   E-mail:
{\tt kmit@informatik.hu-berlin.de}}\\[5mm]
{\small
Institute of Mathematics, Humboldt University of Berlin,}\\
{\small 
Rudower Chaussee 25, D-12489 Berlin, Germany}\\
{\small   E-mail:
{\tt recke@mathematik.hu-berlin.de}}
}
 
\date{}

\begin{document}

\maketitle

\begin{abstract}
\noindent
This paper concerns linear first-order hyperbolic systems in one space dimension 
of the type 
$$
\partial_tu_j  + a_j(x,t)\partial_xu_j + \sum\limits_{k=1}^nb_{jk}(x,t)u_k = f_j(x,t),\;  x \in (0,1),\; j=1,\ldots,n,
$$
with periodicity conditions in time and reflection boundary conditions in space. 
We state a non-resonance condition (depending on the coefficients $a_j$ and  $b_{jj}$ and the boundary 
reflection coefficients), which implies  Fredholm solvability of the problem
in the space of continuous functions. Further, we state one more  non-resonance condition (depending 
also on $\d_ta_j$), which implies $C^1$-solution regularity. Moreover, we give examples showing that both
non-resonance conditions  cannot be dropped, in general.
Those conditions are robust under small perturbations of the problem data.
Our results work for many non-strictly hyperbolic systems, but they are new even in the case 
of strict hyperbolicity.
\end{abstract}



\thispagestyle{empty}

\section{Introduction}\label{sec:intr} 
\renewcommand{\theequation}{{\thesection}.\arabic{equation}}
\setcounter{equation}{0}

\subsection{Problem and main results}\label{sec:results}

This paper concerns general linear first-order
hyperbolic  systems   in one space dimension of the type 
\beq\label{eq:t1}
\partial_tu_j  + a_j(x,t)\partial_xu_j + \sum\limits_{k=1}^nb_{jk}(x,t)u_k = f_j(x,t), \quad x \in (0,1), \; j=1,\ldots,n
\ee
with time-periodicity conditions
\beq\label{eq:t2}
u_j(x,t+2\pi) = u_j(x,t),\quad  x \in [0,1], \; j=1,\ldots,n
\ee
and reflection boundary conditions
\beq\label{eq:t3}
\begin{array}{rcl}
\displaystyle
u_j(0,t) &=&\displaystyle \sum\limits_{k=m+1}^nr_{jk}(t)u_k(0,t),\;j=1,\ldots,m,\\
\displaystyle
u_j(1,t) &=&\displaystyle \sum\limits_{k=1}^mr_{jk}(t)u_k(1,t), \;j=m+1,\ldots,n.
\end{array}
\ee
Here $m$ and $n$ are integers with $0\le m\le n$. Throughout the paper it 
is supposed that the 
functions $r_{jk}: \R \to \R$
and $a_j,b_{jk},f_j:[0,1]\times \R \to \R$ are  $2\pi$-periodic with respect to $t$,
and that the coefficients $r_{jk},a_j$ and $b_{jk}$ are $C^1$-smooth.
Additionally, we suppose that
\beq
\label{ungl}
a_j(x,t)\not=0 \;\mbox{ for all } x \in [0,1],  t \in \R  \mbox{ and }  j=1,\ldots,n
\ee
and that 
\beq
\label{cass}
\begin{array}{l}
\mbox{for all } 1 \le j \not= k \le n \mbox{ there exists }  \tilde b_{jk} \in C^1([0,1]\times \R) \mbox{ such that } \\
b_{jk}(x,t)=\tilde b_{jk}(x,t)(a_k(x,t)-a_j(x,t))  \mbox{ for all }  x \in [0,1]  \mbox{ and }  t \in \R.
\end{array}
\ee

Roughly speaking, we will prove the following:
If the first non-resonance condition  \reff{eq:t8} on the data $a_j, b_{jj}$ and  $r_{jk}$
is satisfied (which is the case, for example, if the functions $|r_{jk}|$ with $1 \le j \le m$ and $m+1 \le k \le n$ or 
with 
 $1 \le k \le m$ and $m+1 \le j \le n$ 
are sufficiently small), then a Fredholm alternative 
is true for the system~(\ref{eq:t1})--(\ref{eq:t3}),
i.e., 
\begin{itemize}
\item
either the system~(\ref{eq:t1})--(\ref{eq:t3}) with $f=(f_1,\ldots,f_n)=0$ has a nontrivial continuous
solution (then the vector space of those solutions has a finite dimension),
\item
or  for any continuous right-hand side $f$ the system~(\ref{eq:t1})--(\ref{eq:t3}) has a unique 
continuous solution $u=(u_1,\ldots,u_n)$ 
(then the map $f \mapsto u$ is continuous with respect to the supremum norm).
\end{itemize}
Moreover, if the second  non-resonance condition \reff{eq:t81}
is satisfied (which is the case  if, for example, the functions $|\d_ta_j|$ are sufficiently small), 
then the solutions to (\ref{eq:t1})--(\ref{eq:t3}) have additional regularity. In particular, if all 
coefficients 
$a_j$ are $t$-independent and if all the data are $C^\infty$-smooth, then all solutions to  (\ref{eq:t1})--(\ref{eq:t3}) 
 are $C^\infty$-smooth. It should be emphasized that  the second non-resonance condition cannot be dropped, in 
 general (Remark 1.4).
Further, we give an example showing that, if the  first  non-resonance condition is not 
satisfied, then it may happen that the Fredholm
solvability is not true (Remark 1.3). Also, we  provide an example showing that, if the assumption \reff{cass} is not 
satisfied, then it may happen that neither Fredholm
solvability nor higher order regularity holds (Remark 1.5).

In order to formulate our results more precisely, let us introduce the characteristics of the hyperbolic system \reff{eq:t1}. 
Given $j=1,\ldots,n$, $x \in [0,1]$, and $t \in \R$, the $j$-th characteristic is defined as the solution 
$\xi\in [0,1] \mapsto \om_j(\xi,x,t)\in \R$ of the initial value problem
\beq\label{char}
\partial_\xi\om_j(\xi,x,t)=\frac{1}{a_j(\xi,\om_j(\xi,x,t))},\;\;
\om_j(x,x,t)=t.
\ee
Moreover,  we denote
\begin{eqnarray}
\label{cdef}
c_j(\xi,x,t)&:=&\exp \int_x^\xi
\frac{b_{jj}(\eta,\om_j(\eta,x,t))}{a_{j}(\eta,\om_j(\eta,x,t))}\,d\eta,\\
\label{ddef}
d_j(\xi,x,t)&:=&\frac{c_j(\xi,x,t)}{a_j(\xi,\om_j(\xi,x,t))}.
\end{eqnarray}
Straightforward calculations (see Section 2) show that a $C^1$-map $u:[0,1]\times \R \to \R^n$ is a solution to 
the PDE problem \reff{eq:t1}--\reff{eq:t3} if and only if
it satisfies the following system of integral equations
\begin{eqnarray}
\label{rep1}
\lefteqn{
u_j(x,t)=c_j(0,x,t)\sum_{k=m+1}^nr_{jk}(\om_j(0,x,t))u_k(0,\om_j(0,x,t))}\nonumber\\
&&-\int_0^x d_j(\xi,x,t)\sum_{k=1\atop k\not=j}^n b_{jk}(\xi,\om_j(\xi,x,t))u_k(\xi,\om_j(\xi,x,t))d\xi\nonumber\\ 
&&+\int_0^x d_j(\xi,x,t)f_j(\xi,\om_j(\xi,x,t))d\xi,\quad j=1,\ldots,m,
\end{eqnarray}
\begin{eqnarray}
\label{rep2}
\lefteqn{
u_j(x,t)=c_j(1,x,t)\sum_{k=1}^m r_{jk}(\om_j(1,x,t))u_k(1,\om_j(1,x,t))}\nonumber\\
&&+\int_x^1 d_j(\xi,x,t)\sum_{k=1\atop k\not=j}^n
b_{jk}(\xi,\om_j(\xi,x,t))u_k(\xi,\om_j(\xi,x,t))d\xi\nonumber\\ 
&&-\int_x^1 d_j(\xi,x,t)f_j(\xi,\om_j(\xi,x,t))d\xi,\quad j=m+1,\ldots,n.
\end{eqnarray}
This motivates the following definition:
\begin{defn}\label{defn:cont}
(i) By $\Cv$ we denote the vector space of all continuous maps $u:[0,1]\times\R \to \R^n$ which  satisfy \reff{eq:t2},
with the norm
\beq
\label{infnorm}
\|u\|_\infty:=\max_{1 \le j \le n} \;\max_{0 \le x \le 1} \;\max_{t \in \R}|u_j(x,t)|.
\ee
(ii) A function $u \in \Cv$ is called a continuous solution to  \reff{eq:t1}--\reff{eq:t3} if it  satisfies \reff{rep1} and \reff{rep2}.

(iii) A function $u\in C^1\left([0,1]\times\R;\R^n\right)$ is called a classical solution to 
 \reff{eq:t1}--\reff{eq:t3} if it satisfies  \reff{eq:t1}--\reff{eq:t3} pointwise.
\end{defn}
Finally, set
$$
\begin{array}{lcl}
\displaystyle
R^0&:=&\displaystyle\max_{m+1\le j \le n}\sum_{k=1}^m \exp \int_0^1\max_{\tau,t\in \R}
\left(\frac{b_{kk}(\eta,t)}{a_k(\eta,t)}-\frac{b_{jj}(\eta,\tau)}{a_j(\eta,\tau)}\right)\,d\eta\,\sum_{l=m+1}^n\max_{\tau,t\in \R}|r_{jk}(\tau)r_{kl}(t)|,
\\\displaystyle
S^0&:=&\displaystyle\max_{1\le j \le m}\sum_{k=m+1}^n \exp \int_0^1\max_{\tau,t\in \R}
\left(\frac{b_{kk}(\eta,t)}{a_k(\eta,t)}-\frac{b_{jj}(\eta,\tau)}{a_j(\eta,\tau)}\right)\,d\eta\,
\sum_{l=1}^m\max_{\tau,t\in \R}|r_{jk}(\tau)r_{kl}(t)|,\\\displaystyle
R^1&:=&\displaystyle\max_{m+1\le j \le n}\sum_{k=1}^m \exp \int_0^1\max_{\tau,t\in \R}
\left(\frac{b_{kk}(\eta,t)}{a_k(\eta,t)}-\frac{b_{jj}(\eta,\tau)}{a_j(\eta,\tau)}
+\frac{\d_ta_{j}(\eta,\tau)}{a_j(\eta,\tau)^2} - \frac{\d_ta_{k}(\eta,t)}{a_k(\eta,t)^2}
\right)\,d\eta\,\\ &&\displaystyle\times
\sum_{l=m+1}^n\max_{\tau,t\in \R}|r_{jk}(\tau)r_{kl}(t)|,\\\displaystyle
S^1&:=&\displaystyle\max_{1\le j \le m}\sum_{k=m+1}^n \exp \int_0^1\max_{\tau,t\in \R}
\left(\frac{b_{kk}(\eta,t)}{a_k(\eta,t)}-\frac{b_{jj}(\eta,\tau)}{a_j(\eta,\tau)}
+\frac{\d_ta_{j}(\eta,\tau)}{a_j(\eta,\tau)^2}- \frac{\d_ta_{k}(\eta,t)}{a_k(\eta,t)^2}
\right)\,d\eta\,\\ &&\displaystyle\times
\sum_{l=1}^m\max_{\tau,t\in \R}|r_{jk}(\tau)r_{kl}(t)|.
\end{array}
$$
 Denote by  $\KK$ the vector space of all  continuous solutions to
\reff{eq:t1}--\reff{eq:t3} with $f=0$. \\

The following theorem is the main result of this paper:

\begin{thm}\label{thm:Fredh} Suppose 
\reff{ungl}, \reff{cass}  and 
\beq
\label{eq:t8}
R^0<1 \mbox{ or } S^0<1.
\ee
Then the following is true:

(i) $\dim \KK <\infty$.

(ii) The vector space  of all  $f \in \Cv$ such that there exists a continuous solution 
to  (\ref{eq:t1})--(\ref{eq:t3}) is a closed subspace of codimension $\dim \KK$ in $\CC_n$.

(iii) Either  $\dim \KK>0$ or for any $f \in \Cv$ there exists exactly one continuous solution $u$
to  (\ref{eq:t1})--(\ref{eq:t3}). In the latter case the map $f \in \Cv \mapsto u \in \Cv$ is continuous.

(iv) Suppose that the functions  
$f_j$ are continuously differentiable with respect to $t$ and 
\beq
\label{eq:t81}
\max \{R^0,R^1\}<1 \mbox{ or }\max \{S^0,S^1\}<1.
\ee
Then any 
continuous solution 
to  (\ref{eq:t1})--(\ref{eq:t3}) is a classical solution to  (\ref{eq:t1})--(\ref{eq:t3}).

(v) If all coefficients $a_j$ are $t$-independent and if all functions $a_j, b_{jk}, f_j$ and $r_{jk}$ are $C^\infty$-smooth, then any continuous solution 
to  (\ref{eq:t1})--(\ref{eq:t3}) is  $C^\infty$-smooth.
\end{thm}

It is well-known that the Fredholm property of the linearization is a key for many local investigations of
time-periodic solutions to nonlinear ODEs and parabolic PDEs. This is the case for Hopf bifurcation, 
for saddle node bifurcation or  period doubling bifurcation of periodic solutions as well as 
for small periodic forcing of stationary or periodic solutions 
(see, e.g. \cite{IJ} for ODEs and
\cite{Ki} for parabolic PDEs).
But almost nothing is known whether similar  results are true for nonlinear (dissipative) hyperbolic PDEs.

The first aim of the present paper is to  make possible
developing   a theory of 
local smooth continuation and  bifurcation of
time-periodic solutions to  nonlinear hyperbolic PDEs. In particular, in \cite{KmRe3}
we applied our results to prove a Hopf bifurcation theorem for  semilinear  hyperbolic PDEs.

The second aim is applications 
to semiconductor laser dynamics \cite{LiRadRe,Rad,RadWu}.
Phenomena like Hopf bifurcation (describing the appearance of selfpulsations of lasers) 
and  periodic forcing of stationary  solutions 
(describing the modulation of stationary laser states by time periodic electric pumping)
and periodic solutions (describing the synchronization of selfpulsating laser states with small time periodic external optical signals,
cf. \cite{BRS,RePe,ReSa,ReSa1})
are essential for many applications of  semiconductor 
laser devices in communication systems.

In \cite{KmRe1} and  \cite{KmRe2} we proved similar to Theorem \ref{thm:Fredh}
results for the autonomous case,
i.e., the case, when the coefficients $a_j, b_{jk}$ and $r_{jk}$ are $t$-independent. There the weak formulation
of the problem  (\ref{eq:t1})--(\ref{eq:t3})
was a system of variational equations, and we used the  method of Fourier series in anisotropic Sobolev spaces
as in~\cite{Vejvoda}.
In the present paper  the weak formulation
of the problem  (\ref{eq:t1})--(\ref{eq:t3}) is the
system \reff{rep1}--\reff{rep2} of integral equations, and we use the method of integration along characteristics in $C$-spaces. 
In  \cite{KmRe1} and  \cite{KmRe2} it is shown that in the autonomous case the non-resonance condition  \reff{eq:t8} 
implies 
 a uniform positive lower bound for the absolute values of the denominators in the Fourier coefficients of the solutions, i.e., no small divisors occur.

It remains an open question whether Theorem \ref{thm:Fredh} admits a generalization  
to higher space dimensions. On the other hand,  it can be generalized to  
problems  with nonlocal terms in the differential equations (\ref{eq:t1})
as well as in the boundary conditions (\ref{eq:t3})
(including nonlocal integral terms and feedback reflection boundary conditions). Then the 
non-resonance conditions  \reff{eq:t8} and  \reff{eq:t81}  should be changed accordingly.
Also, it turns out that  Theorem \ref{thm:Fredh} can be developed for  general linear
  second-order wave equations with Robin boundary conditions, which is a task of a 
forthcoming paper. Again, a nontrivial
question  is how to modify the conditions \reff{eq:t8} and~\reff{eq:t81}.

The paper is organized as follows. 
In Section~\ref{sec:remarks} we comment about the assumptions 
\reff{cass}, \reff{eq:t8} and  \reff{eq:t81}.
In Section~\ref{sec:characteristics} we show that any classical solution to  (\ref{eq:t1})--(\ref{eq:t3}) 
is a continuous solution in the sense of Definition \ref{defn:cont}, and that any  $C^1$-smooth continuous solution
is a classical one.
In Section~\ref{sec:abstract} we  introduce an abstract representation of the system \reff{rep1}--\reff{rep2}. 
Moreover, we show that in the ``diagonal'' case, i.e., if $b_{jk}=0$ for all $j \not= k$,
there exists exactly one continuous solution to  (\ref{eq:t1})--(\ref{eq:t3}) for every $f \in \CC_n$.
The Fredholm alternative  stated in the assertions (i)-(iii) of 
Theorem \ref{thm:Fredh} is proved in Section~\ref{sec:fredh},
while the solution regularity given by  the assertions (iv) and (v) 
is proved in Section \ref{sec:smoothdep1}.

\subsection{Some Remarks} 
\label{sec:remarks}
\begin{rem}
{\bf about the first non-resonance condition  \reff{eq:t8}}\rm\ 
If \reff{eq:t8} is not fulfilled,  then the assertions (i) and (v) of Theorem \ref{thm:Fredh} are not true, in general. 
To show this, let us
consider the following example  
satisfying all but \reff{eq:t8} assumptions of Theorem \ref{thm:Fredh}:
Set $m=1, n=2, a_1(x,t)=-a_2(x,t)=\alpha=\mbox{const}, b_{jk}(x,t)=0, f_j(x,t)=0$, and $r_{12}=r_{21}=1$.
Then the system \reff{rep1}--\reff{rep2} reads 
\begin{eqnarray}
&&u_1(x,t)=u_2(0,t-\alpha x), \label{1}\\
&&u_2(x,t)=u_1(1,t+\alpha (x-1)).\label{2}
\end{eqnarray}
Here $R^0=S^0=1$, i.e.,  \reff{eq:t8} is not satisfied.
Inserting \reff{2} into \reff{1} and putting $x=1$, we get
\beq
\label{3}
u_1(1,t)=u_1(1,t-2\alpha).
\ee
If $\alpha/2\pi$ is irrational, then the functional equation \reff{3} does not have nontrivial continuous solutions. If 
$$
\frac{\alpha}{2\pi}=\frac{p}{q} \mbox{ with } p \in \Z \mbox{ and } q \in \N,
$$
then any $2\pi/q$-periodic  function is a solution to \reff{3}.
In other words,
$$
\dim \KK=\left\{
\begin{array}{cl}
0 & \mbox{ if } \alpha/2\pi \not\in \Q,\\
\infty & \mbox{ if } \alpha/2\pi \in \Q,
\end{array}
\right.
$$
hence, the Fredholm solvability conclusion of Theorem \ref{thm:Fredh} is failed.
Moreover, in the case $ \alpha/2\pi \in \Q$ there exist continuous solutions to  (\ref{eq:t1})--(\ref{eq:t3}) which are not classical one's.
\end{rem}

\begin{rem}
{\bf about the second non-resonance condition  \reff{eq:t81}}\rm\ 
If \reff{eq:t81} is not fulfilled,  then the assertion (iv) of Theorem \ref{thm:Fredh} is not true, in general. 
To show this, let us
consider the following example  
satisfying all but \reff{eq:t81} assumptions of Theorem \ref{thm:Fredh}:
\beq
\begin{array}{ll}
\partial_tu_1  - (2+\sin t)\partial_xu_1  = 0, \quad \partial_tu_2  + \frac{2}{4\pi-1}\partial_xu_2  = 1\nonumber\\
u_j(x,t+2\pi) = u_j(x,t),\quad j=1,2\nonumber\\
u_1(0,t)=u_2(0,t),\quad u_2(1,t)=r(t)u_1(1,t).
\end{array}
\ee
Hence,
$$
\tau_1(\xi,x,t)=A^{-1}(A(t)+\xi-x)  \mbox{ with } A(t):=-2t+\cos t, 
$$
$$
 \tau_2(\xi,x,t)=(\xi-x)\frac{4\pi-1}{2}+t,
$$
and
$$
\d_t\tau_1(\xi,x,t)=\frac{a(t)}{a(\tau_1(\xi,x,t))}  \mbox{ with } a(t):=-2-\sin t.  
$$
Then the system \reff{rep1}--\reff{rep2} reads 
\begin{eqnarray}
&&u_1(x,t)=u_2(0,A^{-1}(A(t)-x)), \label{1a}\\
&&u_2(x,t)=r\left((1-x)\frac{4\pi-1}{2}+t\right)u_1\left(1,(1-x)\frac{4\pi-1}{2}+t\right)+\frac{4\pi-1}{2}.\label{2a}
\end{eqnarray}
Inserting \reff{2a} into \reff{1a}, we get
\beq
\label{3a}
u_1(1,t)=r\left(\frac{4\pi-1}{2}+A^{-1}(A(t)-1)\right)u_1\left(1,\frac{4\pi-1}{2}+A^{-1}(A(t)-1)\right)+\frac{4\pi-1}{2}.
\ee
We have 
$$
t+2\pi=\frac{4\pi-1}{2}+A^{-1}(A(t)-1)\nonumber 
$$
if and only if $A(t)-1=A(t+\frac{1}{2})$ or, the same,
$$
\cos t-\cos \left(t+\frac{1}{2}\right)=-2\sin\left(t+\frac{1}{4}\right)\sin\left(-\frac{1}{4}\right)=0.
$$
This equation has two different solutions $\pi-1/4$ and $2\pi-1/4$. Set $t_0:=\pi-1/4$.
Then \reff{3a} yields
$$
u_1(1,t_0)=r(t_0)u_1\left(1,t_0\right)+\frac{4\pi-1}{2}
$$
and 
\beq
\label{8a}
\d_tu_1(1,t_0)=r(t_0)\d_t\tau_1(0,1,t_0)\d_tu_1(1,t_0)+r^\prime(t_0)\d_t\tau_1(0,1,t_0)u_1(1,t_0).
\ee
Because of 
$$
\d_t\tau_1(0,1,t_0)=\frac{-2-\sin(\pi-1/4)}{-2-\sin(\pi+1/4)}>1,  
$$
 we can choose a smooth $2\pi$-periodic function $r(t)$ such that
\beq
\label{9a}
0 \le r(t)<1 \mbox{ for all } t \in \R
\ee
and
\beq
\label{10a}
r(t_0)=\frac{1}{\d_t\tau_1(0,1,t_0)} \mbox{ and } r^\prime(t_0)\not=0.
\ee
From \reff{9a} it follows that $R^0=\max_t|r(t)|<1$, and
there exists exactly one continuous solution to \reff{3a} and, hence, to \reff{1a}--\reff{2a}.
But this solution is not differentiable at $t=t_0$ because \reff{10a} contradicts to \reff{8a}.
Moreover, we have
\begin{eqnarray*}
S^0&=&\max_t|r(t)|\exp \max_t\frac{a^\prime(t)}{a(t)^2}\\
& \ge & r(t_0)\exp \int_0^1\frac{a^\prime(\tau_1(\eta,1,t_0))}{a(\tau_1(\eta,1,t_0))^2} d \eta\\
& = & r(t_0)\exp \int_0^1\frac{d}{d\eta}\ln a(\tau_1(\eta,1,t_0))d \eta\\
& = & r(t_0)\frac{a(t_0)}{a(\tau_1(0,1,t_0))}
 =  r(t_0)\d_t\tau_1(0,1,t_0)=1,
\end{eqnarray*}
\end{rem}
what means that the condition \reff{eq:t81} is not satisfied.

\begin{rem}
{\bf about the assumption  \reff{cass}}\rm\
Roughly speaking, assumption \reff{cass} means that  a certain loss 
of strict hyperbolicity, caused by leading order coefficients $a_j$ and $a_k$ with $j\not=k$, 
must be compensated by a certain vanishing behavior of the corresponding lower order coefficients $b_{jk}$.

Let us show that, if \reff{cass} is not fulfilled, 
  Theorem \ref{thm:Fredh} is not true, in general. 
With this aim we
consider the following example  
satisfying all but \reff{cass} assumptions of Theorem \ref{thm:Fredh}:
Set $m=1, n=2, a_1(x,t)=a_2(x,t)=1, b_{11}(x,t)= b_{12}(x,t)= b_{22}(x,t)=0,  b_{21}(x,t)=3/2, f_1(x,t)= f_2(x,t)=0$, and $r_{12}^0=r_{21}^1=1/2$.
Then the system \reff{rep1}, \reff{rep2} reads 
\begin{eqnarray*}
&&u_1(x,t)=\frac{1}{2}u_2(0,t-x),\\
&&u_2(x,t)=\frac{1}{2}u_1(1,t-x+1)+\frac{3}{2}\int_x^1u_1(\xi,\xi-x+t) d\xi.
\end{eqnarray*}
It is easy to verify that any continuous $2\pi$-periodic map $U:\R \to \R$ creates a solution
$$
u_1(x,t)=U(t-x),\quad u_2(x,t)=\left(2-\frac{3}{2}x\right)U(t-x)
$$
to this system.
In particular, we have $\dim \KK=\infty$, and  there exist continuous solutions to  (\ref{eq:t1})--(\ref{eq:t3}) which are not classical one's.

Let us remark that, surprisingly,  
 assumptions  of the type \reff{cass} 
are  used also in quite another circumstances,
for proving the spectrum-determined growth condition in $L^p$-spaces \cite{Guo,Luo,Neves} and in $C$-spaces \cite{Lichtner}
for semiflows generated by initial value problems for hyperbolic systems of the type  \reff{eq:t1}, \reff{eq:t3}.
\end{rem}

\section{Integration along characteristics}\label{sec:characteristics}
\renewcommand{\theequation}{{\thesection}.\arabic{equation}}
\setcounter{equation}{0}

In this section we show that a $C^1$-function $u:[0,1] \times \R \to \R^n$
satisfies the differential system \reff{eq:t1}--\reff{eq:t3} if and only if it satisfies the 
 integral system \reff{rep1}-\reff{rep2}.
The type of calculations is well-known, so we do this for the convenience of the reader.

Standard results about initial value problems for ordinary differential equations yield that the functions
$\om_j:[0,1]\times[0,1]\times\R \to \R$ are well-defined by \reff{char}, and they are $C^1$-smooth. Moreover, it holds
\begin{eqnarray}
&&\om_j(\xi,x,t+2\pi)=\om_j(\xi,x,t)+2\pi,\label{21}\\
&&\om_j(x,\xi,\om_j(\xi,x,t))=t\label{22}
\end{eqnarray}
and
\begin{eqnarray}
\label{dx}
\d_x\om_j(\xi,x,t) & = & -\frac{1}{a_j(x,t)} \exp \int_\xi^x 
\frac{\d_ta_j(\eta,\om_j(\eta,x,t))}{a_j(\eta,\om_j(\eta,x,t))^2} d \eta,\\
\label{dt}
\d_t\om_j(\xi,x,t) & = & \exp \int_\xi^x \frac{\d_ta_j(\eta,\om_j(\eta,x,t))}{a_j(\eta,\om_j(\eta,x,t))^2} d \eta
\end{eqnarray}
for all $j=1,\ldots,n$, $\xi, x \in [0,1]$, and $t \in \R$. From \reff{dx} and \reff{dt} it follows
\beq
\label{23}
\left(\d_t+a_j(x,t)\d_x\right)\vphi(\om_j(\xi,x,t))=0
\ee
for all $j=1,\ldots,n$, $\xi,x \in [0,1]$, $t \in \R$ and any $C^1$-function $\vphi: \R \to \R$.

Now, let us show that any $C^1$-solution to  \reff{rep1}-\reff{rep2} is a solution to \reff{eq:t1}-\reff{eq:t3}.
Let $u$ be a  $C^1$-solution to  \reff{rep1}-\reff{rep2}. Then \reff{23} yields
\begin{eqnarray*}
\lefteqn{
\left(\d_t+a_j(x,t)\d_x\right)\left(r_{jk}(\om_j(0,x,t))u_k(0,\om_j(0,x,t))\right)}\\
&&=\left(\d_t+a_j(x,t)\d_x\right)\left(r_{jk}(\om_j(1,x,t))u_k(1,\om_j(1,x,t))\right)\\
&&=\left(\d_t+a_j(x,t)\d_x\right)\left(b_{jk}(\xi,\om_j(\xi,x,t))u_k(\xi,\om_j(\xi,x,t))\right)\label{der_1}\\
&&=\left(\d_t+a_j(x,t)\d_x\right)f_j(\xi,\om_j(\xi,x,t))=0,\nonumber
\end{eqnarray*}
and \reff{cdef}, \reff{ddef}, and \reff{23} imply
$$
\begin{array}{cc}
\left(\d_t+a_j(x,t)\d_x\right)c_j(\xi,x,t) = -b_{jj}(x,t)c_j(\xi,x,t),\\
\left(\d_t+a_j(x,t)\d_x\right)d_j(\xi,x,t) = -b_{jj}(x,t)d_j(\xi,x,t).
\end{array}
$$
Hence, \reff{eq:t1} is satisfied.
The time-periodicity conditions \reff{eq:t2} follow directly from  \reff{rep1}, \reff{rep2}, and \reff{21},
while the boundary conditions \reff{eq:t3} follow from  \reff{rep1}, \reff{rep2}, and \reff{22}.

Now,  let us show that any $C^1$-solution to  \reff{eq:t1}-\reff{eq:t3} is a solution to \reff{rep1}-\reff{rep2}.
Let $u$ be a  $C^1$-solution to  \reff{eq:t1}-\reff{eq:t3}. Then
\begin{eqnarray*}
\lefteqn{
\frac{d}{d\xi} u_j(\xi,\om_j(\xi,x,t))=
\d_xu_j(\xi,\om_j(\xi,x,t))+\frac{\d_tu_j(\xi,\om_j(\xi,x,t))}{a_j(\xi,\om_j(\xi,x,t))}}\\
&&=\frac{1}{a_j(\xi,\om_j(\xi,x,t))}\left(-\sum_{k=1}^n 
b_{jk}(\xi,\om_j(\xi,x,t))u_k(\xi,\om_j(\xi,x,t))+f_j(\xi,\om_j(\xi,x,t))\right).
\end{eqnarray*}
This is a linear inhomogeneous ordinary differential equation for the function $u_j(\cdot,\om_j(\cdot,x,t))$, 
and the variation of constants formula (with initial condition at $x_j$) gives
\begin{eqnarray*}
\lefteqn{
u_j(x,t)=u_j(x_j,\om_j(x_j,x,t))\exp \int_{x_j}^{x}\left(- \frac{b_{jj}(\xi,\om_j(\xi,x,t))}{a_j(\xi,\om_j(\xi,x,t))} \right)d\xi}\\
&&-\int_{x_j}^{x}\exp \int_\xi^x\left(- \frac{b_{jj}(\eta,\om_j(\eta,x,t))}{a_j(\eta,\om_j(\eta,x,t))} \right) d \eta \\
&&\times \sum_{k\not= j} \frac{b_{jk}(\xi,\om_j(\xi,x,t))u_k(\xi,\om_j(\xi,x,t))-f_j(\xi,\om_j(\xi,x,t))}
{a_j(\xi,\om_j(\xi,x,t))} d\xi.
\end{eqnarray*}
Here and in what follows we use the notation
\beq
\label{x_j}
x_j:=\left\{
\begin{array}{rcl}
0 & \mbox{for} & j=1,\dots,m,\\
1 & \mbox{for} & j=m+1,\dots,n.
\end{array}
\right.
\ee
Inserting the boundary conditions \reff{eq:t3} for $j=1,\ldots,m$, 
we get \reff{rep1} and inserting  \reff{eq:t3} for $j=1,\ldots,m$, 
we get \reff{rep2}.

\section{Abstract representation of \reff{rep1}--\reff{rep2}}\label{sec:abstract}
\renewcommand{\theequation}{{\thesection}.\arabic{equation}}
\setcounter{equation}{0}

The system  \reff{rep1}--\reff{rep2} can be written as the operator equation
\beq\label{abstr}
u=Cu+Du+Ff,
\ee
where the linear bounded operators $C,D,F: \CC_n \to \CC_n$ are defined as follows: 

Denote by $\CC_m$ the space of all continuous maps 
$v: [0,1] \times \R \to \R^m$ with $v(x,t+2\pi)=v(x,t)$ for all $x \in [0,1]$ and $t \in \R$, with the norm
$$
\|v\|_\infty:=\max_{1 \le j \le m} \max_{0 \le x \le 1} \max_{t \in \R} |v_j(x,t)|.
$$
Similarly we define the space $\CC_{n-m}$. The spaces $\CC_n$ and $\CC_m \times 
\CC_{n-m}$ will be identified, i.e., elements $u \in \CC_n$ will be written as $u=(v,w)$ with $v \in \CC_m$ and $w \in \CC_{n-m}$.
Define linear bounded operators $K:\CC_{n-m} \to \CC_m$ and  $L:\CC_{m} \to \CC_{n-m}$ by 
\beq
\label{Cdef}
\begin{array}{l}
\displaystyle (Kw)_j(x,t):=
\displaystyle c_j(0,x,t)\sum_{k=m+1}^nr_{jk}(\om_j(0,x,t))w_k(0,\om_j(0,x,t)),\quad  j=1,\ldots,m,\\
(Lv)_j(x,t):= \displaystyle c_j(1,x,t)\sum_{k=1}^m r_{jk}(\om_j(1,x,t))v_k(1,\om_j(1,x,t)),\quad j=m+1,\ldots,n.
\end{array}
\ee
Then the operator $C$ is defined as
\beq
\label{KL1}
Cu:=(Kw,Lv) \mbox{ for } u=(v,w).
\ee
The operators $D$ and $F$ are given by
\begin{eqnarray}
\label{Ddef}
&&(Du)_j(x,t):=-\int_{x_j}^x d_j(\xi,x,t)\sum_{k=1\atop k\not=j}^n b_{jk}(\xi,\om_j(\xi,x,t))u_k(\xi,\om_j(\xi,x,t))d\xi, \\
\label{Fdef}
&&(Ff)_j(x,t):=
\displaystyle  \int_{x_j}^x d_{j}(\xi,x,t)f_j(\xi,\om_j(\xi,x,t))d\xi.
\end{eqnarray}

\begin{lem}
\label{lem:iso}
Suppose 
\reff{eq:t8}. Then $I-C$ is bijective from $\CC_n$ to $\CC_n$.
\end{lem}
\begin{proof}
Let $f=(g,h) \in \CC_n$ with $g \in \CC_m$ and $h \in \CC_{n-m}$ be arbitrary given. 
We have $u=Cu+f$ if and only if
$
v=Kw+g, w=Lv+h,
$
i.e., if and only if
$$
v=K(Lv+h)+g \,\,\mbox{ or }\,\,
 w=L(Kw+g)+h.
$$
Taking into account \reff{Cdef}, 
we  have $u=Cu+f$ if and only if
$$
v(1,t)=[K(Lv+h)+g](1,t) \,\,\mbox{ or }\,\,
w(0,t)=[L(Kw+g)+h](0,t).
$$
Hence, if for some $c<1$ 
\beq
\label{KL<1}
\max_{1 \le j \le m} \max_{t \in \R}|[KLv]_j(1,t)|\le c \max_{m+1 \le j \le n} \max_{t \in \R}|v_j(1,t)| 
\,\,\,\,\mbox{ for all } v\in \CC_{n-m}
\ee
or 
\beq
\label{LK<1}
\max_{m+1 \le j \le n} \max_{t \in \R}|[LKw]_j(0,t)|\le c \max_{1 \le j \le m} \max_{t \in \R}|w_j(0,t)| 
\,\,\,\,\mbox{ for all } w\in \CC_{m}
\ee
then $I-C$ is bijective from $\CC_n$ to $\CC_n$. 
We use \reff{eq:t8} and get 
$$
[KLv]_j(1,t)=c_j(0,1,t)\sum_{k=m+1}^nr_{jk}(\tau_j(0,1,t))[Lv]_k(0,\tau_j(0,1,t))
$$
and
$$
[Lv]_k(0,\tau_j(0,1,t))=c_k(1,0,\tau_j(0,1,t))\sum_{l=1}^mr_{kl}\left(\tau_k(1,0,\tau_j(0,1,t))\right)
v_l\left(1,\tau_k(1,0,\tau_j(0,1,t))\right).
$$
This yields  \reff{KL<1} with $c=S^0$. Similarly one shows  \reff{LK<1} with $c=R^0$.
Hence, assumption \reff{eq:t8} yields \reff{KL<1} or  \reff{LK<1}.
\end{proof}

\section{Fredholm property}\label{sec:fredh}
\renewcommand{\theequation}{{\thesection}.\arabic{equation}}
\setcounter{equation}{0}

In this section we assume \reff{ungl}, \reff{cass} and \reff{eq:t8} and prove  the assertions (i)--(iii) 
of Theorem~\ref{thm:Fredh}.

We have to show that the operator $I-C-D$ is Fredholm of index zero from $\CC_n$ to $\CC_n$.
Unfortunately, the operator $D$ is not compact from  $\CC_n$ to $\CC_n$, in general, because it is a partial integral operator (cf. \cite{Appell}).
But by Lemma~\ref{lem:iso},  the operator $I-C-D$ is Fredholm of index zero from $\CC_n$ to $\CC_n$
if and only if
\beq
\label{Fre}
I-(I-C)^{-1}D \mbox{ is Fredholm of index zero from } \CC_n \mbox{ to } \CC_n,
\ee
and for proving \reff{Fre} we use the following Fredholmness criterion (cf., e.g. \cite[Theorem XIII.5.2]{KA}):

\begin{lem}\label{thm:criter}
Let $W$ be a Banach space, $I$ the identity in $W$, and $A:W \to W$ a linear bounded operator with $A^2$ being
compact.  Then $I+A$ is a Fredholm operator of index zero.
\end{lem}

Now, for \reff{Fre} it is sufficient to show that the operator $(I-C)^{-1}D(I-C)^{-1}D$
is compact from $\CC_n$ to $\CC_n$, i.e., that
\beq
\label{comp}
D(I-C)^{-1}D  \mbox{ is compact from } \CC_n   \mbox{ to } \CC_n.
\ee
Because of $D(I-C)^{-1}D=D^2+DC(I-C)^{-1}D$, the statement \reff{comp} will be proved if we show that
\beq
\label{Fr11}
D^2 \mbox{ and } DC  \mbox{ are compact from } \CC_n \mbox{ to } \CC_n.
\ee

Let us denote by $\CC_n^1$ the Banach space of all $u \in \CC_n$, which are $C^1$-smooth, with the norm
$\|u\|_\infty+\|\d_xu\|_\infty+\|\d_tu\|_\infty$. By the Arcela-Ascoli theorem, $\CC_n^1$ is compactly embedded into $\CC_n$. Hence, 
for \reff{Fr11} it  suffices to show that
\beq
\label{Fr2}
D^2 \mbox{ and } DC  \mbox{ map } \CC_n \mbox{ continuously into } \CC_n^1.
\ee

The definitions \reff{Cdef}, \reff{KL1} and \reff{Ddef} imply that for all 
$u \in \CC_n^1$ we have  $\d_xu,\d_tu \in \CC_n^1$ 
and
\beq
\label{d}
\begin{array}{ll}
\d_xCu=C_{11}u + C_{12}\d_tu, & \d_tCu=C_{21}u + C_{22}\d_tu,\\
\d_xDu=D_{11}u + D_{12}\d_tu, & \d_tDu=D_{21}u + D_{22}\d_tu
\end{array}
\ee
with linear bounded operators $C_{jk}, D_{jk}: \CC_n \to \CC_n$, which are defined by
$$
(C_{11}u)_j(x,t):=\left\{
\begin{array}{ll}
\displaystyle \sum_{k=m+1}^n \d_xc_{jk}(x,t)u_k(0,\om_j(0,x,t)) & \mbox{for } j=1,\ldots,m,\\
\displaystyle \sum_{k=1}^m \d_xc_{jk}(x,t)u_k(1,\om_j(1,x,t)) & \mbox{for } j=m+1,\ldots,n,
\end{array}
\right.
$$
$$
(C_{12}u)_j(x,t):=\left\{
\begin{array}{ll}
\displaystyle \sum_{k=m+1}^n \d_x\om_j(0,x,t)c_{jk}(x,t)u_k(0,\om_j(0,x,t)) & \mbox{for } j=1,\ldots,m,\\
\displaystyle \sum_{k=1}^m \d_x\om_j(1,x,t)c_{jk}(x,t)u_k(1,\om_j(1,x,t)) & \mbox{for } j=m+1,\ldots,n,
\end{array}
\right.
$$
$$
(C_{21}u)_j(x,t):=\left\{
\begin{array}{ll}
\displaystyle \sum_{k=m+1}^n \d_tc_{jk}(x,t)u_k(0,\om_j(0,x,t)) & \mbox{for } j=1,\ldots,m,\\
\displaystyle \sum_{k=1}^m \d_tc_{jk}(x,t)u_k(1,\om_j(1,x,t)) & \mbox{for } j=m+1,\ldots,n,
\end{array}
\right.
$$
$$
(C_{22}u)_j(x,t):=\left\{
\begin{array}{ll}
\displaystyle \sum_{k=m+1}^n \d_t\om_j(0,x,t)c_{jk}(x,t)u_k(0,\om_j(0,x,t)) & \mbox{for } j=1,\ldots,m,\\
\displaystyle \sum_{k=1}^m \d_t\om_j(1,x,t)c_{jk}(x,t)u_k(1,\om_j(1,x,t)) & \mbox{for } j=m+1,\ldots,n
\end{array}
\right.
$$
and 
\begin{eqnarray*}
&&(D_{11}u)_j(x,t):=
\sum_{k=1\atop k\not=j}^n\int_{x_j}^x \d_xd_{jk}(\xi,x,t)u_k(\xi,\om_j(\xi,x,t))d\xi,\\
&&(D_{12}u)_j(x,t):=\sum_{k=1\atop k\not=j}^n\int_{x_j}^x \d_x\om_j(\xi,x,t)d_{jk}(\xi,x,t)u_k(\xi,\om_j(\xi,x,t))d\xi,\\
&&(D_{21}u)_j(x,t):= \sum_{k=1\atop k\not=j}^n\int_{x_j}^x \d_td_{jk}(\xi,x,t)u_k(\xi,\om_j(\xi,x,t))d\xi, \\
&&(D_{22}u)_j(x,t):= \sum_{k=1\atop k\not=j}^n\int_{x_j}^x \d_t\om_j(\xi,x,t)d_{jk}(\xi,x,t)u_k(\xi,\om_j(\xi,x,t))d\xi.
\end{eqnarray*}
Here
\beq
\label{cjkdef}
c_{jk}(x,t):=\left\{
\begin{array}{ll}
\displaystyle r_{jk}(\om_j(0,x,t)) c_j(0,x,t) & \mbox{for } j=1,\ldots,m,\; k=m+1,\ldots,n,\\
\displaystyle  r_{jk}(\om_j(1,x,t)) c_j(1,x,t) & \mbox{for } j=m+1,\ldots,n, \; k=1,\ldots,m
\end{array}
\right.
\ee
and 
\beq
\label{djkdef}
d_{jk}(\xi,x,t):=
-d_j(\xi,x,t) b_{jk}(\xi,\om_j(\xi,x,t)).
\ee
By \reff{d} we get for all $u \in \CC_n^1$ 
$$
\begin{array}{ll}
\d_xD^2u=D_{11}Du + D_{12}(D_{11}u+D_{12}\d_tu), & \d_tD^2u=D_{21}Du + D_{22}(D_{21}u+D_{22}\d_tu),\\
\d_xDCu=D_{11}Cu + D_{12}(C_{11}u+C_{12}\d_tu), & \d_tDCu=D_{21}Cu + D_{22}(C_{21}u+C_{22}\d_tu).
\end{array}
$$
Now, taking into account the density of $\CC_n^1$ in $\CC_n$, in order to show \reff{Fr11} 
it suffices to prove the following statement:
\begin{lem}\label{Ck}
There exists a positive constant such that  for all $u \in \CC_n^1$ we have
$$
\|D_{12}^2\d_tu\|_\infty+\|D_{22}^2\d_tu\|_\infty+\|D_{12}C_{12}\d_tu\|_\infty+\|D_{22}C_{22}\d_tu\|_\infty
\le \mbox{\rm const } \|u\|_\infty.
$$
\end{lem}
\begin{proof}
For any $j=1,\ldots,n$ and $u \in \CC_n^1$ we have
\begin{eqnarray}
\!\!\!\!\!\!\!\!\!\!\!\!\!\!&&(D_{12}^2\d_tu)_j(x,t)\nonumber\\
\!\!\!\!\!\!\!\!\!\!\!\!\!\!&&=\sum_{k=1\atop k\not=j}^n
\sum_{l=1\atop l\not=k}^n
\int_{x_j}^x\int_{x_k}^\xi d_{jkl}(\xi,\eta,x,t)b_{jk}(\xi,\om_j(\xi,x,t))\d_tu_l(\eta,\om_k(\eta,\xi,\om_j(\xi,x,t)))  d\eta d\xi
\label{D^2}
\end{eqnarray}
with
$$
d_{jkl}(\xi,\eta,x,t)
:=\d_t\om_j(\xi,x,t)\d_t\om_k(\eta,\xi,\om_j(\xi,x,t)))
d_{j}(\xi,x,t)d_{kl}(\eta,\xi,\om_j(\xi,x,t)).
$$
On the other hand, \reff{char}, \reff{dx} and \reff{dt} imply (for all $\xi,\eta,x \in [0,1]$ 
and $t \in \R$ with $\d_tu_l(\eta,\om_k(\eta,\xi,\om_j(\xi,x,t))) \not=0$)
\begin{eqnarray}
\lefteqn{
\frac{\frac{d}{d\xi}u_l(\eta,\om_k(\eta,\xi,\om_j(\xi,x,t)))}{\d_tu_l(\eta,\om_k(\eta,\xi,\om_j(\xi,x,t)))}}\nonumber\\
&&=\d_x\om_k(\eta,\xi,\om_j(\xi,x,t))+\d_t\om_k(\eta,\xi,\om_j(\xi,x,t))\d_\xi\om_j(\xi,x,t)\nonumber\\
&&=\d_t\om_k(\eta,\xi,\om_j(\xi,x,t))
\left(\frac{1}{a_j(\xi,\om_j(\xi,x,t))}-\frac{1}{a_k(\xi,\om_j(\xi,x,t))}\right).
\label{djkbruch}
\end{eqnarray}
Hence, \reff{cass}, \reff{dx} and \reff{dt} 
yield that for all $\xi,\eta,x \in [0,1]$ and $t \in \R$ it holds
\begin{eqnarray}
&&d_{jkl}(\xi,x,t)b_{jk}(\xi,\om_j(\xi,x,t))\d_tu_l(\eta,\om_k(\eta,\xi,\om_j(\xi,x,t)))\nonumber\\
&&=\tilde{d}_{jkl}(\xi,\eta,x,t)\tilde{b}_{jk}(\xi,\om_j(\xi,x,t))
\frac{d}{d\xi}u_l(\eta,\om_k(\eta,\xi,\om_j(\xi,x,t)))
\label{tilde}
\end{eqnarray}
with
$$
\tilde{d}_{jkl}(\xi,\eta,x,t):=a_j(\xi,\om_j(\xi,x,t))a_k(\xi,\om_j(\xi,x,t))\d_t\om_j(\xi,x,t)d_j(\xi,x,t)d_{kl}(\eta,\xi,\om_j(\xi,x,t)).
$$
Remark that the values $\tilde{b}_{jk}(x,t)$ are not uniquely defined for $(x,t)$ with $a_j(x,t)=a_k(x,t)$ by the condition \reff{cass},
but, anyway, the right-hand side (and, hence, the left-hand side) of \reff{tilde} does not depend on the choice of  $\tilde{b}_{jk}$ because 
$\frac{d}{d\xi}u_l(\eta,\om_k(\eta,\xi,\om_j(\xi,x,t)))=0$ if $a_j(x,t)=a_k(x,t)$ (cf. \reff{djkbruch}).

Let us check if for all $j\not=k$ and $k\not=l$ the partial derivatives $\d_\xi \tilde{d}_{jkl}$ exist and are continuous: For the factors $a_j(\xi,\om_j(\xi,x,t))$ and $a_k(\xi,\om_j(\xi,x,t))$
this is the case because $a_j$, $a_k$, and $\om_j$ are $C^1$-smooth.
For the factor $\d_t\om_j(\xi,x,t)$
this is the case because $\d_x\om_j$ is $C^1$-smooth (cf.  \reff{ungl} and \reff{dx}).  
Finally, for the factors  $d_j(\xi,x,t)$ and $d_{jk}(\eta,\xi,\tau_j(\xi,x,t))$ 
this follows from \reff{ddef} and \reff{djkdef}. 

Applying Fubini's theorem and integration by parts in \reff{D^2}, we get, for example, 
for the terms  with $1 \le j,k \le m$
\begin{eqnarray*}
\lefteqn{
\left|\int_0^x\int_0^\xi d_{jkl}(\xi,\eta,x,t)b_{jk}(\xi,\om_j(\xi,x,t))\d_tu_l(\eta,\om_k(\eta,\xi,\om_j(\xi,x,t))) d\eta d\xi \right|}\\
&&=\left|\int_0^x\int_\eta^x \tilde{d}_{jkl}(\xi,\eta,x,t)\tilde{b}_{jk}(\xi,\om_j(\xi,x,t))
\frac{d}{d\xi}u_l(\eta,\om_k(\eta,\xi,\om_j(\xi,x,t))) d\xi d\eta\right|\\
&&\le\left|\int_0^x\int_\eta^x \frac{d}{d\xi} \left(\tilde{d}_{jkl}(\xi,\eta,x,t)\tilde{b}_{jk}(\xi,\om_j(\xi,x,t))\right)
u_l(\eta,\om_k(\eta,\xi,\om_j(\xi,x,t))) d\xi d\eta \right|\\ 
&&+\left|\int_0^x\left[\tilde{d}_{jkl}(\xi,\eta,x,t)\tilde{b}_{jk}(\xi,\om_j(\xi,x,t))
u_l(\eta,\om_k(\eta,\xi,\om_j(\xi,x,t))\right]_{\xi=\eta}^{\xi=x}d\eta\right|\\
&&\le \mbox{ const } \|u\|_\infty.
\end{eqnarray*}
Similarly one can handle the other terms in \reff{D^2}. Consequently, we get
$\|D_{12}^2\d_tu\|_\infty \le \mbox{const } \|u\|_\infty$. The estimate $\|D_{22}^2\d_tu\|_\infty \le \mbox{const } \|u\|_\infty$
can be proved in an analogous way.

Further, for any $u \in \CC_n^1$ we have
(using the notation \reff{x_j})
$$
(D_{22}C_{22}\d_tu)_j(x,t)=
\sum_{k=1\atop k\not=j}^n  \sum_{l=1}^n \int_{x_j}^x 
e_{jkl}(\xi,x,t)b_{jk}(\xi,\om_j(\xi,x,t))\d_tu_l(x_j,\om_k(x_j,\xi,\om_j(\xi,x,t))) d\xi
$$
with
\begin{eqnarray*}
\lefteqn{
e_{jkl}(\xi,x,t)}\\
&&:=
\left\{
\begin{array}{l}
-\d_t\om_j(\xi,x,t)\d_t\om_k(1,\xi,\om_j(\xi,x,t))d_{j}(\xi,x,t)c_{kl}(\xi,\om_j(\xi,x,t)) \mbox{ for } l=1,\ldots m,\\
-\d_t\om_j(\xi,x,t)\d_t\om_k(0,\xi,\om_j(\xi,x,t))d_{j}(\xi,x,t)c_{kl}(\xi,\om_j(\xi,x,t)) \mbox{ for } l=m+1,\ldots n.
\end{array}
\right.
\end{eqnarray*}
Using \reff{djkbruch}, we get for $l=1,\ldots,m$
\begin{eqnarray*}
&&e_{jkl}(\xi,x,t)b_{jk}(\xi,\om_j(\xi,x,t))\d_tu_l(1,\om_k(1,\xi,\om_j(\xi,x,t)))\\
&&=\tilde{e}_{jkl}(\xi,x,t)\tilde{b}_{jk}(\xi,\om_j(\xi,x,t))\frac{d}{d\xi}u_l(1,\om_k(1,\xi,\om_j(\xi,x,t)))
\end{eqnarray*}
with
$$
\tilde{e}_{jkl}(\xi,\eta,x,t):=-a_j(\xi,\om_j(\xi,x,t))a_k(\xi,\om_j(\xi,x,t))\d_t\om_j(\xi,x,t)d_j(\xi,x,t)c_{kl}(\eta,\xi,\om_j(\xi,x,t)).
$$
Hence, we can integrate by parts and get
\begin{eqnarray*}
\lefteqn{
\left|\int_{x_j}^x  e_{jkl}(\xi,x,t)b_{jk}(\xi,\om_j(\xi,x,t))\d_tu_l(1,\om_k(1,\xi,\om_j(\xi,x,t))) d\xi\right|}\\
&&=\left|\int_{x_j}^x \tilde{e}_{jkl}(\xi,x,t)\tilde{b}_{jk}(\xi,\om_j(\xi,x,t))\frac{d}{d\xi}u_l(1,\om_k(1,\xi,\om_j(\xi,x,t))) d\xi\right|\\
&&\le \mbox{ const } \|u\|_\infty.
\end{eqnarray*} 
Similarly one can proceed in the case $l=m+1,\ldots,n$. This way we come to the estimate
$\|D_{22}C_{22}\d_tu\|_\infty \le \mbox{const } \|u\|_\infty$.
To get the same upper bound for the remaining term 
$D_{12}C_{12}\d_tu$,
we follow the same line. The proof is therewith complete.
\end{proof}

\begin{rem}
{\bf about smoothness assumptions on the coefficients \boldmath$b_{jk}$\unboldmath}\rm\ 
Concerning the regularity assumptions on the coefficients $b_{jk}$ with $j \not=k$,
we only used (see  the proof of Lemma \ref{Ck}) that
\beq
\label{R1}
\left|\int_{x_j}^x \tilde{b}_{jk}(\xi,\om_j(\xi,x,t))\frac{d}{d\xi}u_l(1,\om_k(1,\xi,\om_j(\xi,x,t))) d\xi\right|
\le \mbox{ const } \|u\|_\infty \mbox{ for all } u \in \CC_n^1.
\ee
This means that the assumption $\tilde{b}_{jk} \in \CC_n^1$ (cf. \reff{cass}) is sufficient, but not necessary. 
For example,
if $a_j$, $a_k$ and $b_{jk}$ and, hence,  $\tilde{b}_{jk}$ are $t$-independent, then for \reff{R1} it is sufficient that
$\tilde{b}_{jk} \in BV(0,1)$.
\end{rem}

\section{Solution regularity}\label{sec:smoothdep1}

\renewcommand{\theequation}{{\thesection}.\arabic{equation}}
\setcounter{equation}{0}
In this section we assume \reff{ungl}, \reff{cass}, \reff{eq:t8} and \reff{eq:t81} and 
prove  the assertions (iv) and (v) of Theorem~\ref{thm:Fredh}, similarly to \cite{kmit}. 
Remark that  \reff{eq:t8} and \reff{eq:t81} are identical if all coefficients $a_j$ are $t$-independent.

To prove the assertion (iv), assume 
that the functions  
$f_j$ are continuously differentiable with respect to $t$.
Let $u$ be a continuous solution to   \reff{eq:t1}--\reff{eq:t3}. We have to 
show that the partial derivatives $\d_xu$ and $\d_tu$ exist and are continuous. 
For that it is sufficient to show that $\d_tu$ exists and is continuous, since then \reff{rep1} and \reff{rep2} imply that also
$\d_xu$ exists and is continuous.

Because of \reff{abstr} we have 
\beq
\label{51}
(I-C)u=D(C+D)u+(I+D)Ff.
\ee
Denote by $\tilde{\CC}_n^1$ the subspace of all $v \in \CC_n$ such that the partial derivative $\d_tv$ exists and is continuous. 
By assumption, $f \in \tilde{\CC}_n^1$. Moreover, by \reff{Ddef} and \reff{Fdef}, the operators $D$ and $F$
map  $\tilde{\CC}_n^1$ into  $\tilde{\CC}_n^1$. Therefore, \reff{Fr2} implies that the right-hand side of 
\reff{51} belongs to   $\tilde{\CC}_n^1$.
Hence, it remains to prove the following fact:
\begin{lem}
\label{5.1}
If for $\tilde{u} \in \CC_n$ and $\tilde{f} \in \tilde{\CC}_n^1$ it holds $\tilde{u}=C\tilde{u}+\tilde{f}$, then $\tilde{u} \in \tilde{\CC}_n^1$.
\end{lem}
\begin{proof}
We proceed as in the proof of Lemma \ref{lem:iso}. In particular, we use the Banach spaces $\CC_m$ and $\CC_{n-m}$ and the linear bounded operators
$K:\CC_{n-m} \to \CC_m$ and $L:\CC_{m} \to \CC_{n-m}$, which are introduced there. Further, by $\tilde{\CC}_m^1$ we denote the space of all
$v \in \CC_m$ such that the partial derivatives $\d_tv$ exist and are continuous. 
Similarly the space  $\tilde{\CC}_{n-m}^1$ is
introduced. 

Let $\tilde{u} \in \CC_n$ and $\tilde{f} \in \tilde{\CC}^1_n$ be given and satisfy
\beq 
\label{ab}
\tilde{u}=C\tilde{u}+\tilde{f}.
\ee 
Then $\tilde{u}=(\tilde{v},\tilde{w})$ with certain $\tilde{v} \in  \tilde{\CC}_m$ and  $\tilde{w} \in  \tilde{\CC}_{n-m}$,
$\tilde{f}=(\tilde{g},\tilde{h})$ with certain $\tilde{g} \in  \tilde{\CC}^1_m$ and  $\tilde{h} \in  \tilde{\CC}^1_{n-m}$,
$
\tilde{v}=K\tilde{w}+\tilde{g}, \mbox{ and } \tilde{w}=L\tilde{v}+\tilde{h}.
$
Hence,
$$
\tilde{v}(1,t)=[K(L\tilde{v}+\tilde{h})+\tilde{g}](1,t) \mbox{ and }
\tilde{w}(0,t)=[L(K\tilde{w}+\tilde{g})+\tilde{h}](0,t).
$$
We have to show that $\tilde{u} \in  \tilde{\CC}^1_n$, i.e., that $\tilde{v} \in  \tilde{\CC}^1_m$ and $\tilde{w} \in  \tilde{\CC}^1_{n-m}$.
Due to  \reff{Cdef}, this is equivalent to 
\beq
\label{glatt}
\tilde{v} \in  \tilde{\CC}^1_m \mbox{ or } \tilde{w} \in  \tilde{\CC}^1_{n-m}.
\ee
From \reff{Cdef}) it follows that \reff{glatt} is equivalent to 
\beq
\label{glatt1}
\tilde{v}(1,\cdot) \in  {C}^1_m(\R) \mbox{ or } \tilde{w}(0,\cdot) \in  {C}^1_{n-m}(\R),
\ee
where ${C}^1_m(\R)$ and $ {C}^1_{n-m}(\R)$ are the spaces of $2\pi$-periodic  continuously differentiable funtions from $\R$ into $\R^m$ and $\R^{n-m}$, respectively.
For any $\ga>0$, the space  ${C}_{m}^1(\R)$ is a  Banach spaces with the norm
$$
\|v\|_\ga:=\|v\|_\infty+\ga\|v'\|_\infty \mbox{ with }
\|v\|_\infty:=\max_{1\le j \le m}\max_{t \in \R}|v_j(t)|,
$$
and similarly for   ${C}_{n-m}^1(\R)$.
Hence, if for some $c<1$ and for some $\ga >0$ we have
\beq
\label{in}
\|[KLv](1,\cdot)\|_\ga \le c \|v(1,\cdot)\|_\ga
 \mbox{ for all } v\in \tilde\CC_{m}^1
\ee
or
\beq 
\label{in1}
\|[LKw](0,\cdot)\|_\ga \le c \|w(0,\cdot)\|_\ga
 \mbox{ for all } w\in \tilde\CC_{n-m}^1,
\ee
then $I-C$ is bijective from $\tilde\CC_n^1$ to $\tilde\CC_n^1$, as desired. 

Let us suppose that $R^0<1$ and $R^1<1$. Fix $c$ with
\beq
\label{Rc}
\max\{R^0,R^1\} < c < 1
\ee
and  prove \reff{in} (the proof of  \reff{in1} under the conditions $S^0<1$ and $S^1<1$ 
follows the same line).

Let us calculate $\d_tKLv$. Similarly to \reff{d}, we have for all $v \in \tilde{\CC}_m^1$ and 
$w \in \tilde{\CC}_{n-m}^1$ 
\beq
\label{KLeq}
\d_tKw=K_1w+K_2\d_tw, \quad\d_tLv=L_1v+L_2\d_tv
\ee
with linear bounded operators $K_1,K_2:\CC_{n-m}\to\CC_m$ and $L_1,L_2:\CC_m \to \CC_{n-m}$ defined as follows (cf. \reff{KL1}
and \reff{d}):
\begin{eqnarray*}
(K_1w)_j(x,t)& := &  \sum_{k=m+1}^n \d_tc_{jk}(x,t)w_k(0,\om_j(0,x,t)),\\
(K_2w)_j(x,t)& := &  \sum_{k=m+1}^n \d_t\om_j(0,x,t)c_{jk}(x,t)w_k(0,\om_j(0,x,t)),\\
(L_1v)_j(x,t)& := &  \sum_{k=1}^m \d_tc_{jk}(x,t)v_k(1,\om_j(1,x,t)),\\
(L_2v)_j(x,t)& := &  \sum_{k=1}^m \d_t\om_j(1,x,t)c_{jk}(x,t)v_k(1,\om_j(1,x,t)).
\end{eqnarray*}
Take  $v \in \tilde{\CC}_{n-m}^1$. From \reff{KLeq} it follows 
\begin{eqnarray*}
\lefteqn{
\|[\d_tKLv](1,\cdot)\|_\infty=\|[(K_1Lv+K_2(L_1v+L_2\d_tv)](1,\cdot)\|_\infty}\\
&&\le\|[(K_1L+K_2L_1)v](1,\cdot)\|_\infty
+\|[K_2L_2\d_tv](1,\cdot)\|_\infty.
\end{eqnarray*}
Moreover, using \reff{cdef}, \reff{dt} and \reff{cjkdef},  we get 
\begin{eqnarray*}
\lefteqn{
[K_2L_2\d_tv]_j(1,t)=\sum_{k=m+1}^n\sum_{l=1}^mr_{jk}(\sigma_j(0,t))r_{kl}(\sigma_{jk}(1,t))\d_tv_l(1,\sigma_{jk}(1,t))}\\
&&\times\exp\int_0^1\left(\frac{b_{kk}(\eta,\sigma_{jk}(\eta,t))}{a_k(\eta,\sigma_{jk}(\eta,t))}-\frac{b_{jj}(\eta,\sigma_{j}(\eta,t))}{a_j(\eta,\sigma_{j}(\eta,t))}\right)d\eta\\
&&\times\exp\int_0^1\left(\frac{\d_ta_j(\eta,\sigma_{j}(\eta,t))}{a_j(\eta,\sigma_{j}(\eta,t))^2}-\frac{\d_ta_k(\eta,\sigma_{jk}(\eta,t))}{a_k(\eta,\sigma_{jk}(\eta,t))^2}
\right)d\eta
\end{eqnarray*}
with
$
\sigma_j(\eta,t):=\tau_j(\eta,1,t),\; \sigma_{jk}(\eta,t):=\tau_k(\eta,0,\tau_j(0,1,t)).
$
Hence, we get
\beq
\label{510}
\|[\d_tKLv](1,\cdot)\|_\infty \le \|[K_1Lv+K_2L_1v](1,\cdot)\|_\infty+R^1\|\d_tv(1,\cdot)\|_\infty.
\ee
Moreover, in  the proof of Lemma \ref{lem:iso} we showed that
$$
\left\|[KLv](1,\cdot)\right\|_\infty \le R^0\|v(1,\cdot)\|_\infty.
$$
Finally, choose $\ga$ so small that 
$$
\|[(K_1L+K_2L_1)v](1,\cdot)\|_{\infty}\le \frac{c-R^0}{\ga}\|v(1,\cdot)\|_\infty
\quad \mbox{ for all } v \in \tilde{\CC}_m^1. 
$$
Then  \reff{Rc} and \reff{510}  imply \reff{in}.
\end{proof}
The proof of the  assertion (iv)  of Theorem~\ref{thm:Fredh} is therewith complete.\\

To prove the assertion (v) of Theorem~\ref{thm:Fredh}, suppose that all coefficients $a_j$ are $t$-independent. 
Then \reff{dt} yields that $\d_t\tau_j(\xi,x,t)=1$.
Therefore in \reff{d} we have 
\beq
\label{CD}
C_{22}=C \mbox{ and }  D_{22}=D. 
\ee

Let $u$  be a continuous solution   to  \reff{eq:t1}--\reff{eq:t3}, i.e., a solution to \reff{rep1}--\reff{rep2}, and suppose that all functions 
$a_j, b_{jk},f_j$ and $r_{jk}$ are $C^\infty$-smooth. 

First we show by induction that all partial derivatives $\d_t^ku, k=1,2,\ldots$
exist and are continuous.

For $k=1$ this follows from assertion (iv) of Theorem~\ref{thm:Fredh}.

Now suppose that all  partial derivatives $\d_tu,\ldots,\d_t^ku$ exist and are continuous. Then as in  \reff{d} one gets (cf. \reff{CD})
$$
\d_t^kCu=\sum_{j=0}^{k-1}C_j\d_t^ju+C\d_t^ku, \; \d_t^kDu=\sum_{j=0}^{k-1}D_j\d_t^ju+D\d_t^ku
$$
with linear bounded operators $C_j,D_j:\CC_n \to \CC_n$ such that $C_jv ,D_jv \in \tilde{\CC}_n^1$ for all $ v \in  \tilde{\CC}_n^1$. 
Hence, from 
$(I-C-D)u=Ff$ it follows
$$
(I-C-D)\d_t^ku=\d_t^kFf-\sum_{j=0}^{k-1}\left(C_j\d_t^ju+D_j\d_t^ju\right)=:R_k \in \tilde{\CC}_n^1
$$
and, consequently,
$$
(I-C)\d_t^ku=D\left((C+D)\d_t^ku+R_k\right)+R_k  \in \tilde{\CC}_n^1.
$$
By Lemma \ref{5.1}, $\d_t^ku \in \tilde{\CC}_n^1$, i.e.,  $\d_t^{k+1}u$ exists and is continuous.

Finally we show  that the partial derivatives $\d_x^k\d_t^lu$
exist and are  continuous for all  $k,l\in \N$.
From \reff{eq:t1} it follows
\beq
\label{dxu}
\d_xu_j(x,t)=\frac{1}{a_j(x)}\left(f_j(x,t)-\d_tu_j(x,t)-\sum_{k=1}^nb_{jk}(x,t)u_k(x,t)\right).
\ee
All partial derivatives with respect to $t$  of the right-hand side (and, hence,   of the left-hand side) 
of \reff{dxu} exist and are continuous,
i.e., $\d_x\d_t^lu_j$ exist and are continuous for all $l\in \N$. Therefore the partial derivative with respect to $x$  
of the right-hand side (and, hence,   of the left-hand side) of \reff{dxu} exists and is continuous, and we have
\beq
\label{d^2xu}
\d^2_xu_j=\frac{1}{a_j^2}\left(a_j\d_x\left(f_j-\d_tu_j-\sum_{k=1}^nb_{jk}u_k\right)-\d_xa_j\left(f_j-\d_tu_j-\sum_{k=1}^nb_{jk}u_k\right)\right).
\ee
Again, all partial derivatives with respect to $t$  of the right-hand side of \reff{d^2xu} exist and are continuous.
Hence, $\d^2_x\d_t^lu_j$ exist and are continuous for all $l\in \N$. Therefore the partial derivative with respect to $x$  of the right-hand side of \reff{d^2xu} exists and is continuous,
i.e., $\d_x^3u_j$ exists and is continuous.
By continuation of this procedure  we get the claim.

\section*{Acknowledgments}
The first author  was supported by the Alexander von Humboldt Foundation.  
Both authors acknowledge support of the DFG Research Center {\sc Matheon}
mathematics for key technologies (project D8).

\end{document}